\documentclass{article}

\usepackage{fancyvrb}
\usepackage{calc}
\usepackage{tikz}
\usepackage{listings}
\usepackage{mathtools}

\makeatletter
\def\paragraph{\@startsection{paragraph}{4}%
	\z@\z@{-\fontdimen2\font}%
	{\normalfont\bfseries}}
\makeatother

\usepackage{pgffor}
\usepackage{tikz-cd}
\usepackage{booktabs}
\usepackage{enumerate}

\usepackage{amssymb, amsmath, amsthm}

\usepackage{xypic}% % % % % % % % % % % % % % % % % % % % % %
\usepackage{scalerel}%%%%%%%%%%%%%%%%%%%%%%%%%%%%%%%%%%%%%%%%%%%%%%%%%%%%%%%%%%%%%%%%%%%%%%%%%%%%
\usepackage{ulem}

\usepackage{graphicx}              % to include figures
\usepackage{xkeyval}               % for passing more than 9 arguments to newcommand
\usepackage{amssymb}
\usepackage{color}
\usepackage{breqn}                 % for multiline \left( \right)
\usepackage{bm}                    % for bold mathematical symbol 

\allowdisplaybreaks[1]
\usepackage{nicefrac}

\newcommand{\RN}[1]{%
	\textup{\uppercase\expandafter{\romannumeral#1}}%
}

\allowdisplaybreaks

\newtheorem{thm}{Theorem}[section]

\newcommand{\RR}{\mathbb{R}}      % for Real numbers
\newcommand{\ZZ}{\mathbb{Z}}      % for Integers

\newcommand{\mat}[4]{\left[\begin{smallmatrix*}[r]
		#1 & #2 \\
		#3 & #4 \\
	\end{smallmatrix*}\right]}

\def\CC{{\mathbb C}}

\def\HH{{\mathbb H}}

\def\PP{{\mathbb P}}
\def\QQ{{\mathbb Q}}
\def\RR{{\mathbb R}}

\def\ZZ{{\mathbb Z}}

\DeclareMathOperator{\Gfun}{G}

\def\<{\langle}
\def\>{\rangle}

\DeclareMathOperator{\im}{Im}

\DeclareMathOperator{\SL}{SL}

\numberwithin{equation}{section}

\begin{document}

\title{Some Interesting Connections!}
\author{Alok Shukla} %%leave blank in initial submission to allow for double blind reviewing
%\author{ }

\maketitle

One source of beauty in mathematics is totally unexpected connections between two fundamentally different objects. For instance, is it not surprising that the time period of a real simple pendulum is linked with a function arising out of finding the number of ways in which a positive integer could be decomposed as a sum of two squares? Why should inherent properties and interrelations among counting numbers should appear in the laws of nature that govern the motion of a simple pendulum? In this article we will see some such surprising and beautiful results coming from combinatorics, number theory and physics. 

\section{Alternating permutations problem}
We are familiar with trigonometric functions such as $ \tan x $ and $ \sec x $. Perhaps, there is nothing about them that can surprise us. Let's see. \\
\textbf{The problem:} Find all the permutations of $ n $ numbers such that the second number is larger than the first one, the third number is smaller than the second, the fourth larger than the third, and so on. Thus there is a up, down, up, down an alternating zig-zag  pattern.  For example, 
for $ n=3 $, we have two permutations: $ (1,3,2) $ and $ (2,3,1) $. The number of alternating permutations for $ n = 1, 3, 5,\ldots, 13 $ are as follows
$$ 1, 2, 16, 272, 7936, 353792, 22368256.  $$ 

Interestingly, in 1881 the French
mathematician D\'esir\'e  Andr\'e found (page 2, \cite{Flajolet}) that these numbers appear in the taylor series expansion of $ \tan x $:
\[ \tan x = x + \frac{2}{3!}  x^3 + \frac{16}{5!}x^5 + \frac{272}{7!}x^7 + \frac{7936}{9!}x^9 +\frac{353792}{11!}x^{11}  + \frac{22368256}{13!}x^{13} \ldots \quad . \]
%\begin{figure}
%	\begin{center}
%	\includegraphics[scale=0.1]{fig1.png}
%	\caption{Alternating permutations: For $n=3$. }
%\end{center}
%\end{figure}
\noindent \textbf{Solution (for $ \bf n = $ odd):} Let $ T_n $ denote the number of alternate permutations with $ n $ assumed to be a positive odd integer. We note that  a recursive relation could be obtained by fixing the highest node and attaching $ T_l $ and $ T_r $ permutations from the left and right respectively. We have the following configuration ($ T_l $, max, $ T_r $) that leads to the recursive relation 
\begin{equation}\label{Eq1}
T_n = \sum\limits_{k=1, k \text{ odd }}^{n-2} \binom{n-1}{k} \, T_k T_{n-1-k}.
\end{equation} 
Define the exponential generating function $ T(s) $ as
\begin{equation}\label{Eq2}
T(s) = \sum\limits_{n=1, n \text{ odd }}^{\infty} T_n \, \frac{s^n}{n!}
\end{equation}
It follows from \eqref{Eq1}  and \eqref{Eq2} that 
\begin{equation} \label{Eq3}
\frac{dT(s)}{ds} = 1 + T(s)^2
\end{equation}
It follows that $ T(s) = \tan s $.

It can be proved in a similar fashion that for even $ n $ the number of alternating permutation is connected with
$ \sec x $. \textbf{This is combinatorial trigonometry!}

\section{Gauss, AM \& GM, elliptic integrals and Jacobi theta functions}
Gauss' remarkable computational abilities and legendary skills in manipulating infinite series led him to a certain hypergeometric function and its associated second order hypergeometric differential equation. Moreover this hypergeometric function is also closely connected with a certain elliptic integral and therefore to elliptic curves and modular forms. It is instructive to follow his original arguments.

Gauss discovered the arithmetico-geometric mean (agm) when he was $ 15 $. He started with two numbers $ a $ and $ b $ and wrote $ a_1 = \dfrac{a+b}{2}$  for the arithmetic mean and $b_1 = \sqrt{ab} $   for the geometric mean of $ a $ and $ b $. He then created sequences $ \{a_n\} $ and $ \{b_n\} $  of arithmetic and geometric means by defining $ a_n = \dfrac{a_{n-1}+b_{n-1}}{2}$ and $ b_n = \sqrt{a_{n-1} b_{n-1}}$ for $ n \geq 1 $ with $ a_0 =a $ and $ b_0 =b $. It is not hard to see that the sequences $ \{a_n\} $ and $ \{b_n\} $ converge to a common limit, known as the agm of $ a $ and $ b $, which Gauss denoted by $ M(a,b) $. It is clear that $ M(\alpha a,\alpha b) = \alpha M(a,b)  $. It is also obvious that $ M(1+x,1-x) $ is an even function. Gauss assumed that its reciprocal has an infinite series expansion.
\begin{equation}
\frac{1}{M(1+x,1-x)} = \sum_{k=0}^{\infty} A_k x^{2k}.
\end{equation}
Now, the substitution $ x = \dfrac{2t}{1+t^2} $ leads to
\[
\frac{1}{M(1+x,1-x)} = \frac{1+t^2}{M((1+t)^2,(1-t)^2)} = \frac{1+t^2}{M(1+t^2,1-t^2)}
\]
giving the relation
\[
\sum_{k=0}^{\infty} A_k \left(\frac{2t}{1+t^2}\right)^{2k} = \sum_{k=0}^{\infty} A_k t^{4k}.
\]
Gauss determined the coefficients $ A_k $ from the above relation and thus he obtained the following infinite series development for $ M(1+x,1-x)^{-1} $
\begin{equation}\label{Eq_gauss_hypergometric}
y:= M(1+x,1-x)^{-1} = 1 + \left(\frac{1}{2}\right)^2 x^2 + \left(\frac{1.3}{2.4}\right)^2 x^4 + \left(\frac{1.3.5}{2.4.6}\right)^2 x^6 + \cdots \quad .
\end{equation}
Further it satisfies the following differential equation
\begin{equation}\label{Eq_gauss_differential}
(x^3-x)\, \frac{d^2y}{dx^2} + (3 x^2 -1 )\, \frac{dy}{dx} + xy =0.
\end{equation}
Gauss then noticed a remarkable connection between $ M(1+x,1-x) $ and the complete elliptic integral of the first kind defined as 
\begin{equation}
K(x) := \int_{0}^{\frac{\pi}{2}} \frac{1}{\sqrt{1 - x^2 \sin^2 \phi}} \, d\phi .
\end{equation} 
\begin{thm}{\textbf{(Gauss).}} \label{Thm_gauss}
	Assume $ |x| < 1 $. Then
	\[	 K(x) = {\frac{\pi}{2}}\, \frac{1}{ M(1+x,1-x)}. \]
\end{thm}
\begin{proof}
	The result follows by expanding $ (1 - x^2 \sin^2 \phi)^{-\frac{1}{2}} $ using the binomial theorem, integrating the resulting series term by term and then making use of \eqref{Eq_gauss_hypergometric}.
\end{proof}

It is already amazing to see the creation of a mathematical theory with beautiful interconnections from a seemingly innocent looking idea of agm. However, this was just the beginning. In fact, the general hypergeometric series is defined as 
\begin{equation} \label{Eq_general_hyperbolic_series}
F(\alpha,\beta,\gamma,x) := 1 + \frac{\alpha \beta}{\gamma}\frac{x}{1!} + \frac{\alpha(\alpha +1) \beta(\beta+1)}{\gamma(\gamma+1)}\frac{x^2}{2!} + \cdots \quad , 
\end{equation} 
and it is a solution of the general hypergeometric differential equation defined by Gauss as 
\begin{equation}\label{Eq_gauass_general_hypergeo_differential}
x (1-x) \, \frac{d^2y}{d x^2} + \left[ \gamma - (\alpha + \beta +1)x\right] \, \frac{dy}{dx} - \alpha \beta y = 0.
\end{equation}
It is readily seen that on substituting $ z =x^2 $ the equation $ \eqref{Eq_gauss_differential} $ reduces to a special case of the general hypergeometric differential equation defined by $ \eqref{Eq_gauass_general_hypergeo_differential} $ with $ \alpha =\beta = \frac{1}{2} $ and $ \gamma =1 $.

The function $ M(1,x) $ is also connected with the well known Jacobi theta function. We define for $ q = e^{2 \pi i z} $ with $ \im(z) >0 $,
\begin{align*}
%\theta_1(q) & =  2 \sum_{n=0}^{\infty} (-1)^n q^{(n+ \frac{1}{2})^2} \\
\theta_2(q) &=  2 \sum_{n=0}^{\infty} q^{(n+ \frac{1}{2})^2} \\
\theta_3(q)   &= \sum_{n=-\infty}^{\infty} q^{n^2} \\
\theta_4(q)  & = \sum_{n=-\infty}^{\infty} (-1)^n q^{n^2} 
\end{align*}

$ \theta_3(q) $ is a generating function for the number of ways of representing a number as a sum of squares. More precisely, 
\[
\theta_3^k(q) = \sum_{n \geq 0} r_k(n) q^n
\]
where, $$ r_k(n) = \# \{(x_1,\cdots x_k)\in \ZZ^k \ | \ x_1^2 + \cdots x_k^2 = n  \}, $$ i.e.,  $ r_k(n) $ denotes the number of ways of decomposing the integer $ n $ as a sum of $ k $ squares. 
It can be shown that
\begin{align}
\frac{\theta_3^2(q) + \theta_3^2(q)}{2}  =   \theta_3^2(q^2) \text{ and }    \sqrt{\theta_3^2(q)  \theta_4^2(q)} =\theta_4^2(q^2),
\end{align}	
from which it follows that 
\begin{align}
M(\theta_3^2(q),\theta_4^2(q)) = 1.
\end{align}
We note the following identity (refer page 467, \cite{Whittaker} )
\begin{align} \label{eq_theta}
\theta_2^4(q) + \theta_4^4(q)  = \theta_3^4(q) .
\end{align}
Let $$ x =  \frac{\theta_2^2(q) }{\theta_3^2(q) }.$$ Then we obtain
\begin{align*}
M(1+x,1-x) &= M\left(1 + \frac{\theta_2^2(q) }{\theta_3^2(q) },1-\frac{\theta_2^2(q) }{\theta_3^2(q) }\right) \\
& = \frac{1}{\theta_3^2(q)} M\left(\theta_3^2(q) + \theta_2^2(q) ,\theta_3^2(q) - \theta_2^2(q)\right) \\
& = \frac{1}{\theta_3^2(q)} M\left(\theta_3^2(q), \sqrt{\theta_3^4(q) - \theta_2^4(q) }\right) \\
&= \frac{1}{\theta_3^2(q)} M\left(\theta_3^2(q),\theta_4^2(q)\right) = \frac{1}{\theta_3^2(q)}. \\
\end{align*}
Then from Theorem \ref{Thm_gauss} and the above calculation it follows that
\begin{align} \label{Eq_theta_elliptic}
K(x) & = {\frac{\pi}{2}}\, \frac{1}{ M(1+x,1-x)} = \frac{1 }{2}\pi\theta_3^2(q).
\end{align}
One common underlying theme in the development of modern mathematics is solution of equations: polynomial equations and the symmetry of their roots resulted in Galois theory, a similar study for the roots of differential equations led to the development of Lie theory and the beautiful subject of algebraic geometry also has its roots in the study of roots of equations! The study of hypergeometric differential equation motivated and provided an impetus for development of a large body of important mathematics such as the theory of complex analysis and also of Riemann surfaces. The solutions of hypergeometric differential equations are messy multivalued functions. New methods and concepts such as that of analytic continuation were invented to systematically study and tame these beasts! 

\section{A Simple Pendulum and Number Theory}
In the previous section we discussed some of the important functions in number theory. Now we will see how the Jacobi theta function, that is related to the number of ways in which an integer could be decomposed as a sum of squares, magically appears in the equation describing the time period of a real pendulum. 
\begin{figure}
	\begin{center}
	\includegraphics[scale=0.5]{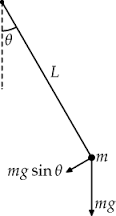}
	\caption{A simple pendulum. }
		\end{center}
\end{figure}
Suppose our simple pendulum consists of a point mass $ m $ attached to one end of a light rod of negligible mass and length $ L $. We assume that the other end of the rod is attached to a frictionless hinge about which the rod is free to swing. Let the rod makes an angle $ \theta_0 $ with the vertical at the highest point of its motion. Let $ T $ be the time period of the pendulum. The conservation of energy, when the angle is $ \theta $, gives
\begin{align*}
mg(L-L \cos \theta_0)  & = m g(L- L \cos \theta) + \frac{1}{2} m \left(L \frac{d \theta}{dt}\right)^2 \\
\implies \frac{d \theta}{dt} & = \sqrt{\frac{2g}{L} \left( \cos \theta - \cos \theta_0 \right)} = \sqrt{\frac{2g}{L} \left( 2 \sin^2 \left(\frac{\theta_0}{2}\right)- 2 \sin^2 \left(\frac{\theta}{2}\right) \right)} \\
\implies 2 \sqrt{\frac{g}{L}}\int_{0}^{\frac{T}{4}} \, dt & = \int_{0}^{\theta_0} \frac{d \theta}{\sqrt{  \sin^2 \left(\frac{\theta_0}{2}\right)-  \sin^2 \left(\frac{\theta}{2}\right)}}\\ &= \int_{0}^{\frac{\pi}{2}} \frac{2 }{\sqrt{  1 -  k^2 \sin^2 \phi }} \, d \phi\\
& = K(k).
\end{align*} 
Here in the second-last step we have changed the variable of integration using the substitution $$ \sin \phi = \frac{\sin \frac{\theta}{2}}{k} \,\, \text{with} \,\, k = \sin \frac{\theta_0}{2}.$$ 
Therefore we obtain the following expression for the time period of the simple pendulum 
\begin{align}
T = 4 \sqrt{\frac{L}{g}} \, K(k) = 2 \pi \sqrt{\frac{L}{g}} M(1+k,1-k)^{-1} = 2 \pi \sqrt{\frac{L}{g}}\,F(\frac{1}{2},\frac{1}{2},1,k) .
\end{align} 
We see the nice connection that the time period of the simple pendulum is given by a function that is related to the arithmetic and geometric means of numbers. From \eqref{Eq_theta_elliptic} it is clear that the time period $ T $ is also related to the Jacobi theta function.

\section{A modular connection}

Let us once again return to the elliptic integral $ K(x) $. These integrals also arose naturally in calculating the arc length of an arbitrary ellipse. Legendre studied these integrals extensively and Jacobi and Abel are credited for coming up with the idea of studying the inverse functions of the elliptic integrals instead.  The entries in Gauss' personal diaries reveal that he had already discovered most of Jacobi's and Abel's results but chose not to publish them. The other important idea was to study the complex valued functions instead of real valued functions. In this situation we can draw a parallel with the familiar circular trigonometric functions as follows. A circle with equation $ y^2 = 1 -x^2  $ is parametrized by $ x = \sin u $ and $ y = \sin^{\prime} u = \cos u $ and $ u =\sin^{-1} x $ is given by the familiar integral $ \int_{0}^{x} \dfrac{1}{\sqrt{1 -t^2}} \, dt $. On extending to the complex domain the $  \sin^{-1} x $ function defined by the above integral is a multivalued function. For example one can choose a path for integration from $ 0 $ to $ x $ on the complex plane such that it goes around an arbitrary number of times around the singularities $ \pm 1 $. However, if we study the inverse of the function defined by the above integral we get a nice single valued periodic function $ \sin x $. The periodicity of $ \sin x  $ is a manifestation of the multivalued nature of the integral defining  $ \sin^{-1} x  $. Similarly it turns out that the multivalued nature of elliptic integrals lead to inverse functions which are doubly periodic. 
\begin{figure}
	\begin{center}
		\includegraphics[scale=0.3]{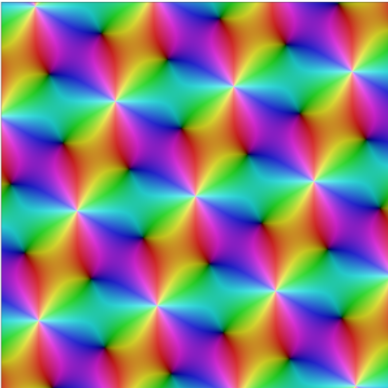}
		\caption{A doubly periodic pattern. }
	\end{center}
\end{figure}

After Gauss, Abel, and Jacobi, Weierstrass was the mathematician to make significant contributions in the study of elliptic integrals with his $ \displaystyle \wp (z) $ function.  In fact, similar to the example of the parametrization of a circle considered earlier, $ x =  \displaystyle \wp (z) $ and $ y =  \displaystyle \wp^{\prime} (z)$ parametrize the  elliptic curve $E(\CC): y^2 = 4 x^3 -g_2 x -g_3 $, with 
\[
z = \int_{x}^{\infty} \frac{dt}{\sqrt{4t^3 - g_2 t - g_3}}.
\]
The function $\displaystyle \wp (z)  $  has many remarkable properties. Suppose $ \Lambda  $ is a lattice, i.e., a subgroup of the form $ \Lambda = \ZZ \,
\omega_1 + \ZZ \, \omega_2 $ with $ \{\omega_1,\omega_2\} $ being an $ \RR $-basis for $ \CC $. A meromorphic function on $ \CC $ relative to the lattice $ \Lambda $ is called an elliptic function if $ f(z + \omega_i) = f(z) $ for all $ z \in \CC $ and $ i \in \{1,2\} $. It turns out that $ \displaystyle \wp (z) $ is an even elliptic function. In fact, there is another formulation of $ \displaystyle \wp (z) $ given below.
\[
\wp (z) = \frac{1}{z^2} + \sum_{\omega \in \Lambda - \{0\}} \left( \frac{1}{(z-\omega)^2} - \frac{1}{\omega^2}\right),
\]
with 
\[g_2 = 60 \Gfun_4(\Lambda), \quad g_3 = 140 \Gfun_6(\Lambda)\] and 
\begin{align}\label{Ch2_def_Eisenstein}
\Gfun_{2k}(\Lambda) := \sum_{\omega \in \Lambda - \{0\}}  \omega^{-2k}.
\end{align}
One can define a map 
\[
\phi: \CC / \Lambda \longrightarrow E(\CC) \subset \PP^2(\CC), z \longleftrightarrow [\wp (z),\wp^{\prime}(z),1]
\]
which is an isomorphism of Riemann surfaces that is also a group homomorphism. From the previous discussion it is clear that now we have a map from lattices to elliptic curves. Further,
\[
\CC / \Lambda \simeq  \CC / \Lambda^{\prime} \iff \Lambda = c \Lambda^{\prime} \quad \text{for some } c \in \CC. 
\]

Also on writing $ \Lambda(\tau) :=  \ZZ \cdot
\tau+ \ZZ \, \cdot 1  \simeq \ZZ \cdot\omega_1 + \ZZ \cdot \omega_2 $ with $ \tau = \frac{\omega_1 }{\omega_2} $ for some $ \tau \in \HH
$, we see that 
$\Lambda(\tau) = \Lambda(\tau^{\prime})  $ if and only if there exist a matrix $ \mat{a}{b}{c}{d}  \in \SL(2,\ZZ)$ such that 
\[
\tau^{\prime} = \frac{a \tau + b}{c \tau + d}.
\]
One can show that there is one-to-one correspondence between  $$ \SL(2,\ZZ) \backslash \HH  \leftrightarrow  \text{elliptic curves over } \CC / \simeq .$$

We now define a modular form of weight $ k $, with respect to a congruence subgroup $ \Gamma $ of  $ \SL(2,\ZZ)$, as a holomorphic function on the complex upper half plane $ \HH $, 

that satisfies the condition
\begin{equation}\label{Eq_modular_form}
f\left(\frac{az+b}{cz+d}\right) = (cz+d)^{k} f(z) \qquad \qquad \text{for }  \mat{a}{b}{c}{d} \in \Gamma, 
\end{equation}
and which is holomorphic at the cusp $ \infty $.
It can be checked that the Eisenstein series $ \Gfun_{2k}(\Lambda) $ defined in \eqref{Ch2_def_Eisenstein} is a modular form of weight $ 2k $. 

We have given a very brief historical account of the origin of modular forms, but many important connections are already beginning to show up. In fact, since their first appearance in connection with hypergeometric equations, modular forms have become a fertile meeting ground for various subfields of mathematics such as elliptic curves,  quadratic forms, quaternion algebras, Riemann surfaces, algebraic geometry, algebraic topology, to name but a few, with fruitful consequences. One such example is the celebrated modularity theorem (formerly Taniyama, Shimura, Weil conjecture), which states that every elliptic curve defined over $ \QQ $ is modular. The proof of a version of the modularity theorem by Wiles resulted in the resolution of Fermat's last theorem, the single most famous problem in number theory. The classical ideas of modular forms were further extended in a more general setting using representation theoretic tools by Gelfand,  Ilya Piatetski-Shapiro and others in the 1960s resulting in the modern theory of automorphism forms. Langlands created a general theory of Eisenstein series and produced some far reaching and deep conjectures. Automorphic forms play an important role in modern number theory. In this context the following remark of Langlands on automorphic forms comes to mind- ``It is a deeper subject than I appreciated and, I begin to suspect, deeper than anyone yet appreciates. To see it whole is certainly a daunting, for the moment even impossible, task.''. 

\subsection*{Summary}
There exist many surprising connections within mathematics and also between mathematics and our physical universe. We could only give glimpses of a few of such interesting results. Readers are welcome to explore and discover many more exciting connections!

\vfill\eject

\end{document}